\documentclass{article}
\usepackage[utf8]{inputenc}

\usepackage{amsmath}
\usepackage{amsfonts}
\usepackage{amsthm}
\usepackage{amssymb}
\usepackage{mathtools}
\usepackage{biblatex}

\usepackage{geometry}
\usepackage[english]{babel}
\newtheorem{theorem}{Theorem}[section]
\newtheorem{corollary}{Corollary}[theorem]
\newtheorem{lemma}[theorem]{Lemma}

\title{Prime-Bounded Subwords}
\author{O.M. Cain \\ onnomawcain@gmail.com}
\date{December 2019}

\begin{document}

\maketitle

\textbf{Abstract.} In the number $373$ all subwords ($3$, $7$, $37$, $73$, and $373$) are prime. Similarly, in $9719$ all subwords are \textit{divisible by} at most one prime. And similarly again in $7319797913$ all subwords are divisible by at most two primes. These are the largest integers with their respective properties. We show for any $k\ge 1$ there are only finitely many integers having subwords divisible by at most $k$ primes. In fact, we show for any $B$ and $d$ coprime that $n$ contains a base-$B$ subword divisible by $d$ if $n>B^d$. So as example consequence, past a certain point every prime contains a subword divisible by say $10000000007$.

\section{Examples}

The author thinks it best to start with appetizers before getting at the main course. The only integers with all subwords prime are 
$$2,\ 3,\ 5,\ 7,\ 23,\ 37,\ 53,\ 73,\ \text{and }373.\text{ (A085823 [1])}$$
We may loosen our restriction requiring instead the subwords to be \textit{divisible by} at most one prime (so for example, $1$, $8=2^3$, $25=5^2$, ... are now allowed). There are $56$ such numbers:
$$1,\ 2,\ 3,\ 4,\ 5,\ 7,\ 8,\ 9,\ 11,\ ...,\ 2719,\ 3137,\ 3797,\ \text{and }9719.$$

We may loosen our restriction again requiring the subwords to be divisble by at most \textit{two} distinct primes. There are $9993$ of these numbers. The largest three of which are
$$981319193,\ 4713191939,\text{ and }7319797913.$$
We can cook the same dish in other bases too. The largest ternary integers with subwords divisible by at most $1,2,3,...$ primes are:
$$71=2122|_3$$
$$19655=222221222|_3$$
$$387243341=222222122222222222|_3$$
$$...$$

\section{Strictly Prime Subwords or \textit{Substrimes}}

For our first course, we'll bound the integers whose subwords are primes. We should clarify what exactly we mean by ``subwords" before continuing. 

Given some integer with base-$B$ expansion $n=x_m...x_0|_B$ we mean by ``subwords of $n$" the function
$$W(i,j)=\sum_{k=i}^jx_kB^{k-i}$$
evaluated at $0\le i\le j\le m$. So for example if $n=7319797913$ then $W(1,5)=79791$, $W(0,0)=W(8,8)=3$, $W(7,9)=731$, and so on.

This first lemma patches a small but surprisingly troublesome hole in the following theorem. The author could find no elegant proof and hopes a more skilled reader can offer an alternative route.

\begin{lemma}
    For an integer $B\ge 3$ the smallest prime not dividing $B$ is less than $B$.
\end{lemma}
\begin{proof}
Suppose some $B$ is divisible by all primes less than itself. Letting $\pi(n)$ count primes less than $n$ we get
$$B\ge\prod_{\substack{p\text{ prime}\\p < B}} p\ge 2^{\pi(B)}.$$

We will use the ``classical Chebyshev bound"\footnote{Taken from: \texttt{https://math.stackexchange.com/questions/1890741}}
$$\pi(B)> \frac{B}{2\log_2{B}}$$
which holds true for $B\ge 3$. Putting the two together gives
$$B\ge 2^{B/(2\log_2{B})}.$$
The right-hand-side grows faster and the expression becomes contradiction for $B>4$. Thus we have only two cases to check with bare hands $B=3$ (in which case $p=2$ will do) and $B=4$ (in which case $p=3$ will do).
\end{proof}

The following theorem was first hinted (though not given explicitly) by ``Jakob B."\footnote{See: \texttt{https://math.stackexchange.com/questions/3048309}}

\begin{theorem}
  If $n$ has all base-$B$ subwords prime then $n<B^{2p}$ where $p$ is the smallest prime not dividing $B$.
\end{theorem}
\begin{proof}
  Suppose instead $n\ge B^{2p}$ -- or equivalently that $n$ has at least $2p+1$ digits in its base-$B$ expansion. Since each subword can take only one of $p$ values mod $p$ we can conclude by pigeon-holing there must be $a<b<c$ in the range $[0,2p]$ such that
  $$W(0,a)\equiv W(0,b)\equiv W(0,c)\mod p.$$
  And so it follows 
  $$W(0,c)-W(0,a)=\sum_{k=a+1}^c x_kB^k=B^{a+1}W(a+1,c)\equiv 0 \mod p.$$
  In other words $p\mid B^{c+1}W(a+1,c)$ and since $p$ by definition does not divide $B$ it must be that $p\mid W(a+1,c)$. But the subwords are assumed prime so $W(a+1,c)=p$.
  
  Next we use the fact that $c>a+1$ to bound $p$ from below:
  $$p=W(a+1,c)=\sum_{k=a+1}^c x_kB^{k-(a+1)}\ge x_cB.$$
  And since $x_c=W(c,c)$ is prime we must have $x_c>1$ so $p>B$. But we can suppose $B\ge 3$ since there are no single-digit primes in base $2$ -- and therefore no such $n$ as was presumed. The previous lemma accordingly guarantees us $p<B$; a contradiction.
  
\end{proof}

A graphical version of Theorem 2.2 was also produced:
\begin{center}
\includegraphics[scale=0.18]{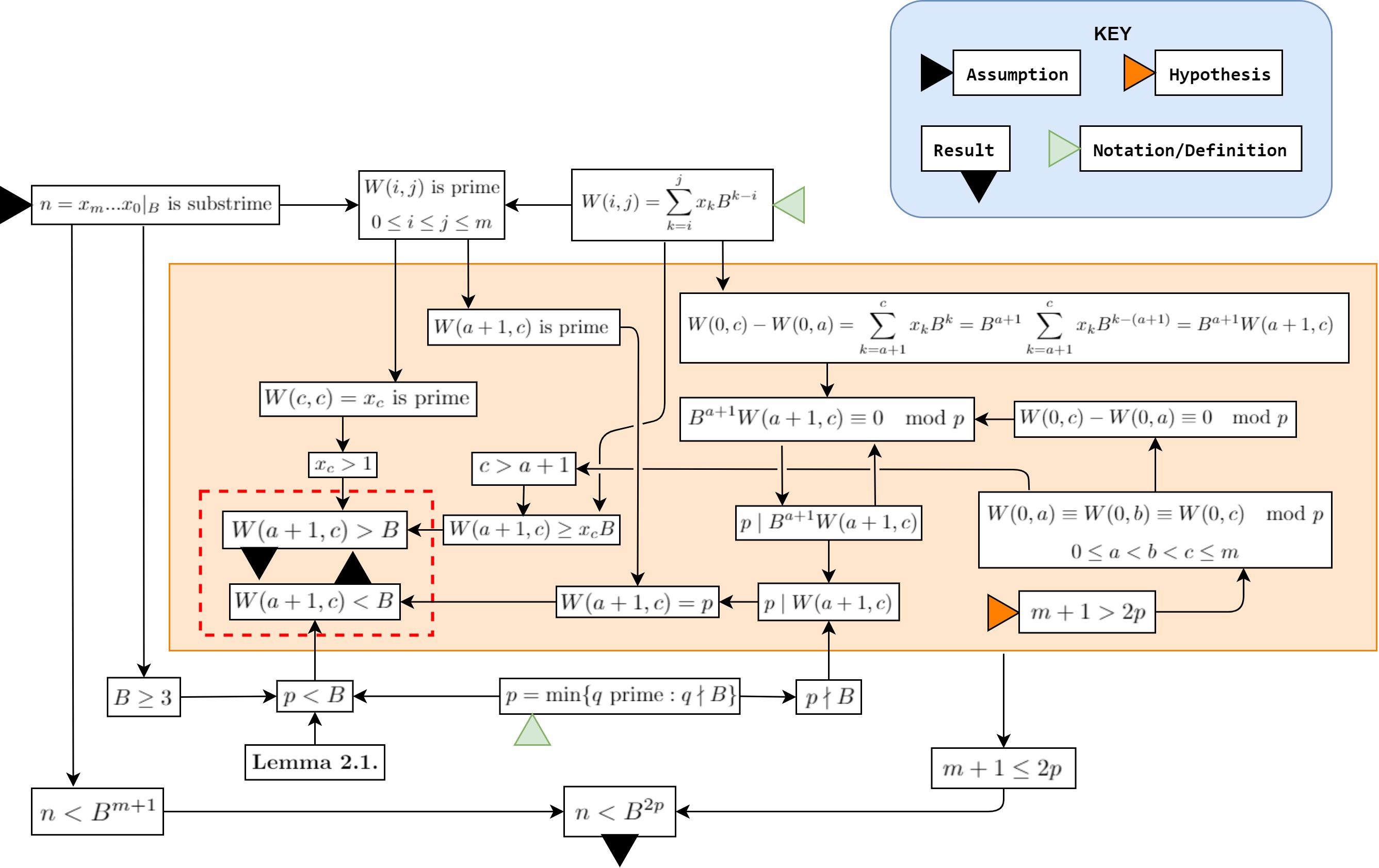}
\end{center}

The term "substrime" -- a \textit{substring-prime} -- was coined to refer to primes with all subwords prime.

The largest substrimes for the first few bases were calculated:
\begin{center}
\begin{tabular}{c|c|c}
    $B$ & largest substrime & total no. substrimes \\ \hline
    $3$ & $2$ & $1$ \\
    $4$ & $11=23|_4$ & $3$ \\
    $5$ & $67=23|_5$ & $5$ \\
    $6$ & $23=23|_6$ & $5$ \\
    $7$ & $37=23|_7$ & $7$ \\
    $8$ & $491=23|_8$ & $19$ \\
    $9$ & $47=23|_9$ & $7$ \\
    $10$ & $373$ & $9$ \\
    $11$ & $79=72|_{11}$ & $6$ \\
    $12$ & $6043=35(11)7|_{12}$ & $25$ \\
     & ... &  \\
    $24$ & $266003=(19)5(19)(11)|_{24}$ & $103$ \\
     & ... &  \\
    $30$ & $485504623=(19)(29)(11)(19)(17)(13)|_{30}$ & $161$ \\
     & ... &  \\
    $90$ & $5495055221=(83)(67)(71)(79)(11)|_{90}$ & $1455$ \\
    $91$ & $8101=(19)(2)|_{91}$ & $35$ \\
     & ... &  \\
\end{tabular}
\end{center}

\section{Prime-bounded Subwords}

We proceed to looser restrictions of subwords.

\begin{theorem}
If $n\ge B^d$ for integers $B$ and $d$ coprime, then at least one base-$B$ subword of $n$ is divisible by $d$.
\end{theorem}
\begin{proof}
If $n\ge B^d$ then $n$ has at least $d+1$ digits in its base-$B$ expansion. And since each subword can take on only one of $d$ values mod $d$ we are assured by pigeon-holing at least two subwords of the form $W(0,j)$ are equal mod $d$ for $0\le j\le d$. Say
$$W(0,a)\equiv W(0,b)\mod d$$
for $a<b$. Thus
$$W(0,b)-W(0,a)=\sum_{k=a+1}^b x_kB^k=B^{a+1}W(a+1,b)\equiv 0 \mod d.$$
In other words $d\mid B^{a+1}W(a+1,b)$. But we defined $B$ and $d$ coprime so $d\mid W(a+1,b)$.
\end{proof}

\begin{corollary}
For any $k\ge 1$ there are finitely many integers whose subwords are divisible by at most $k$ primes.
\end{corollary}
\begin{proof}
Apply the theorem setting $d$ to the product of the first $k$ primes not dividing $B$ say $p_1,...,p_k$. Thus any such integer satisfies $n < B^{\prod p_i}$.
\end{proof}

\end{document}